\documentclass[11pt, oneside, letterpaper, reqno]{amsart}

\usepackage{amsthm, amssymb, amscd, amsmath}
\usepackage{graphicx, color}
\usepackage{mathrsfs}

\usepackage{enumerate}
\usepackage{stmaryrd}
\usepackage[foot]{amsaddr}

\usepackage[hidelinks, colorlinks=true, linkcolor = blue, linktoc = page,
citecolor = blue]{hyperref}

\usepackage[T1]{fontenc}
\usepackage{lmodern}

\usepackage[text={5.4in, 8.0in},centering]{geometry}
\usepackage{microtype}

\newtheorem{theorem}{Theorem}[section]  
\newtheorem*{theorem*}{Theorem}
\newtheorem{proposition}[theorem]{Proposition}

\newtheorem{lemma}[theorem]{Lemma}

\theoremstyle{definition}
\newtheorem{remark}{Remark}
\newtheorem{definition}{Definition}

\newcommand{\cL}{\mathcal{L}}
\newcommand{\cLh}{\mathcal{L}_{h}}
\newcommand{\st}{:\,}
\newcommand{\id}{\mathrm{id}}
\newcommand{\T}{\mathbb{T}}
\newcommand{\cD}{\mathcal{D}}
\newcommand{\cB}{\mathcal{B}}
\newcommand{\cP}{\mathcal{P}}
\newcommand{\R}{\mathbb{R}}
\newcommand{\cA}{\mathcal{A}}
\newcommand{\cG}{\mathcal{G}}
\newcommand{\cS}{\mathcal{S}}
\newcommand{\cO}{\mathcal{O}}

\newcommand{\cJ}{\mathcal{J}}
\newcommand{\C}{\mathbb{C}}
\DeclareMathOperator{\dist}{\mathrm{dist}}

\newcommand{\Z}{\mathbb{Z}}
\newcommand{\N}{\mathbb{N}}
\newcommand{\bS}{\mathbb{S}}
\newcommand{\bfc}{\mathbf{c}}

\newcommand{\bmat}[1]{\begin{bmatrix}#1\end{bmatrix}}

\newcommand{\floor}[1]{\lfloor #1 \rfloor}
\renewcommand{\Re}{\mathrm{Re}}

\title[Generic a priori unstable global diffusion]{Generic global diffusion for analytic uncoupled a priori unstable systems}

\author{Amadeu Delshams$^1$}
\address{$^1$Laboratory of Geometry and Dynamic Systems and IMTech, Universitat Politècnica de Catalunya (UPC) and 
Centre for Mathematical Research (CRM), Barcelona }

\author{Ke Zhang$^2$}
\address{$^2$University of Toronto}

\begin{document}

\begin{abstract}
We show that given a general uncoupled a priori unstable Hamiltonian
\[
\frac12 p^2 + V(q) + G(I) + \epsilon h(p, q, I, \varphi, t),
\]
where $h$ is a generic Ma\~n\'e analytic function and $\epsilon$ is small enough, there is an orbit for which the momentum $I$ changes by any arbitrarily prescribed value. We call this phenomenon as global diffusion since the size of the change in $I$ is independent of both $\epsilon$ and $h$. The fact that the pendulum and rotor variables are uncoupled is used essentially in our proof. The proof is based on simple and constructive geometrical methods, carefully studying the reduced Poincar\'e functions of the problem which generate the corresponding scattering maps.
\end{abstract}

\maketitle
\section{Introduction}

For a nearly integrable Hamiltonian
\begin{equation}  \label{eq:nearly-integrable}
  H_\epsilon(I, \varphi) = H_0(I) + \epsilon H_1(I, \varphi),
\end{equation}
a large part of the phase space is occupied by invariant KAM tori, making the system remarkably stable. Arnold diffusion is the study of topological instability in the complement of these KAM tori. The instability is created by resonances of the unperturbed frequency $\partial_I H_0(I)$. Indeed, in the first example of Arnold diffusion, Arnold~\cite{Arn64} considers the system
\begin{equation}  \label{eq:arnold}
  H_\epsilon = \frac12(I_1^2 + I_2^2) + \epsilon (\cos \varphi_1 - 1) + \epsilon \mu  (\cos \varphi_1 - 1) f(\varphi_2, t),
\end{equation}
and proves that for suitable $f$ and  $0<\mu \ll \epsilon$, there exists an orbit whose $I_2$ component change for an arbitrary large distance along the resonance $\{I_1 = 0\}$. Since $\mu$ is taken much smaller than $\epsilon$, after a rescaling we may assume $\epsilon = 1$ in \eqref{eq:arnold}. This type of systems are called \emph{a priori unstable} because the hyperbolic structure that leads to instability exists even at $\mu = 0$. The original problem \eqref{eq:nearly-integrable} is called \emph{a priori stable}, since the unperturbed system is completely integrable.

We have seen much progress in proving  Arnold diffusion holds for \emph{generic} $C^r$ \emph{a priori stable} systems, we refer to \cite{Mat04, Mat08, Che17, Che19, Mar16, Mar24, KZ20} for a non-exhaustive list of references. Those results are built on decades of progress on \emph{a priori unstable} systems, see \cite{Ber08, CY04, CY09, DLS06, DSd08, Tre02, Tre04, KZB16, GM22} and reference therein.

For real analytic systems, however, generic Arnold diffusion for \emph{a priori stable} systems remain wide open. A large body of literature is dedicated to analytic examples of diffusion (for example \cite{Zha11, KSL14}) and concrete systems like the N-body problems (for example \cite{DKRS19,CFG22}). For generic analytic perturbation of \emph{a priori chaotic} systems, we mention the works of Gelfreich-Turaev~\cite{TG17} and Clarke-Turaev~\cite{TC22}. For a priori unstable systems, Delshams-Schaefer~\cite{DS17, DS18} studied the diffusion for a finite parameter family of perturbations.

The most relevant work to this paper is \cite{LC22} by Chen-de la Llave, who proved that for an \emph{a priori unstable} system of the type
\begin{equation}  \label{eq:chen-llave}
\begin{aligned}
		H_\epsilon & = H_0(p, q, I) + \epsilon h(p, q, I, \varphi, t), \quad
H_0(p, q, I) = G(I) + \sum_{i = 1}^d \pm \left( \frac12 p_i^2 + V_i(q_i)\right), \\
							 & \quad p \in \R^d, \, q \in \T^d, \, I \in \R, \varphi, t \in \T,
\end{aligned}
\end{equation}
for a Ma\~n\'e generic analytic perturbation $H_1$ (precise definition to be given later), there exists a diffusion orbit $(p, q, I, \varphi)(t)$ and $T > 0$ such that
\[
  |I(T) - I(0)| > \rho(H_1).
\]
The point is that the magnitude of topological instability does not vanish as $\epsilon \to 0$.

The main result of our paper is that in a system of type \eqref{eq:chen-llave} one can achieve \emph{large-scale} diffusion for a  Ma\~n\'e generic analytic perturbation independent of $H_1$. Indeed, our main result claims that given \emph{any} $I^- < I^+$, for a  Ma\~n\'e generic analytic potential $H_1$, \emph{for all} $\epsilon < \epsilon_0(H_1)$, the system $H_0 + \epsilon H_1$ admits an orbit $(p, q, I, \varphi)(t)$ and $T > 0$ such that
\[
  I(0) < I^- < I^+ < I(T).
\]
We call this large drift of the variable $I$ \emph{global diffusion}, compared to the local-scale diffusion of \cite{LC22}.
In particular, it follows from our result that Arnold's Theorem in \cite{Arn64} remains true for a generic analytic perturbation of the type $ h(\varphi_1, \varphi_2, t)$ instead of the special perturbation $(\cos \varphi_1 - 1) f(\varphi_2, t)$ chosen by Arnold.

Let us point out that while the unperturbed Hamiltonian of the type \eqref{eq:chen-llave} are often considered synonymous to \emph{a priori unstable} system, it is quite special because the variables $I$ and $p$ are uncoupled, as is the case in \emph{all} the mentioned a priori unstable systems. This has a significant effect in the \emph{scattering map} approach to Arnold diffusion, because the uncoupling of the variables $I$, $p$ implies that the scattering map has \emph{zero phase shift}. Indeed, our proof does \emph{not} apply directly if the unperturbed Hamiltonian is
\begin{equation}  \label{eq:phase-shift}
		H_0(p, q, I) = G(I) + \sum_{i = 1}^d \pm \left( \frac12 p_i^2 + V_i(q_i)\right) + I \sum_{i = 1}^d \lambda_i p_i
\end{equation}
for a non-zero vector $\lambda = (\lambda_1, \dots, \lambda_d)$. Generic global diffusion for a priori unstable analytic systems in the presence of phase shift remains open. To avoid any possible confusion, we have included in the title of this paper the adjective `uncoupled' in the a priori unstable systems that we are dealing with.

Our proof uses the scattering map approach (see \cite{DLS06, DSd08, GL05, dGM20}). The diffusion orbit travels close to a normally hyperbolic invariant manifold (NHIM), and the dynamics can be
described by the scattering map (homoclinic excursion) and inner map (dynamics restricted to the NHIM). The scattering map can be computed to the first order by means of the Melnikov potential $\cL_h(I,\varphi,s)$ introduced in~\eqref{eq:melnikov}. Our approach is twofold: 
\begin{itemize}
		\item Generically, for every fixed $I$ the Melnikov potential $\cL_h(I, \cdot, \cdot)$ admits a non-degenerate local minimum. These local minima are obtained from local extension of the global minima. 
		\item Near these local minima of the Melnikov potential, we show that there exists a pseudo-orbit of the scattering map and the restricted dynamics to the invariant manifold whose $I$ component is increasing. We call this the ``Ascending Ladder''. We then show that there exists a trajectory of the Hamiltonian system which shadows the ascending ladder. 
\end{itemize}
In the approach followed, no quantitative estimates of diffusion time are made nor is there an attempt to optimize them, since our goal is to propose a very general, simple and constructive diffusion mechanism. In the pursuit of simplicity for the reader, this article is almost self-contained, except for the theory of scattering maps and geometric mechanism of diffusion (Proposition \ref{prop:homoclinic}, equation \eqref{eq:SM1}, and Lemma \ref{lem:shadowing-primitive}), for which specific references are given.


\subsection*{Acknowledgment}
AD supported by Spanish grant PID2021-123968NB-I00 (MICIU/AEI/10.13039/501100011033/FEDER/UE). KZ is supported by NSERC fund RGPIN-2019-07057.

\section{Formulation of the main result}
Consider the Hamiltonian
\begin{equation}  \label{eq:unstable}
  H_\epsilon(p, q, I, \varphi, t) = \frac12 p^2 + V(q) + G(I) + \epsilon h(p, q, I, \varphi, t),
\end{equation}
where
\[
		(p, q, I, \varphi, t) \in M := \R^d \times \T^d \times \R \times \T \times \T
\]
and $p^2$ stands for $p \cdot p$.

Assume that
\begin{enumerate}
		\item $V$ has a non-degenerate global maximum; without loss of generality, suppose it is at $q = 0$.
		\item For the Hamiltonian system $\frac12 p^2 + V(q)$, the stable and unstable manifolds of $(0, 0)$ intersect along a homoclinic trajectory
        \begin{equation}
\label{eq:homoclinic}
						\left\{\left(p_0(t), q_0(t)\right) : \ t \in \R\right\}.
				\end{equation}
		\item $G(I)$ is strictly convex in $I$.
\end{enumerate}
Let us denote $\omega(I) = \nabla G(I)$.
\begin{remark}
While we require $G(I)$ to be strictly convex, the proof is not variational. The only essential property used is the twist property $\omega'(I) \ne 0$, hence the proof also applies to concave $G(I)$, or on each interval of $I$ for which $\omega'(I)$ does not vanish.
\end{remark}
\begin{definition}
\label{def:spaces}
Let $U \subset M$ be an open set. We define its $\sigma$-complex neighborhood to be
\[
		\cB_\sigma(U) =
		\{
				\tilde z = (p, q, I, \varphi, t) \in \C^d \times (\C/\Z)^d \times \C \times (\C/\Z) \times (\C/\Z) \st
				\dist (\tilde z, U) < \sigma
		\}.
\]
Let $\cA_\sigma(U)$ denote the space of all real analytic functions on $U$ that can be extended to a bounded analytic function on $\cB_\sigma(U)$. Equipped with the supremum norm, this is a Banach space. Let
\[
  Q = \T^d \times \T \times \T
\]
denote the configuration space, i.e., the space of the variables $(q, \varphi, t)$. We will also consider the complex strip
\[
		\cD_\sigma = \{(q, \varphi, t) \in (\C/\Z)^d \times (\C/\Z) \times (\C/\Z) \st \|\Re(q, \varphi, t) \| < \sigma\}
\]
and the space  $\cP_\sigma$  of real analytic functions on $Q$ extensible to a bounded function on $\cD_\sigma$. Identifying a function $f \in \cP_\sigma$ with its trivial extension to $\cB_\sigma(M)$, $\cP_\sigma \subset \cB_\sigma(U)$ for any $U$.
\end{definition}
\begin{theorem}
Let $U \subset M$ be open.
Let $I^- < I^+$ be such that
\[
		\{(p_0(t), q_0(t)) \st t \in \R\} \times [I^-, I^+] \times \T \times \T \subset U.
\]

Then for any $h \in \cA_\sigma(U)$, there exists an open and dense set $\cG_h$ of $\cP_\sigma$ such that, for any $g\in\cG_h$,
the following holds for the Hamiltonian system
\[
    H_\epsilon = \frac12 p^2 + V(q) + G(I) + \epsilon \left( h(p,q, I,\varphi, t) + g(q, \varphi, t) \right).
\]
There exists $\epsilon_0(h+g) > 0$ such that for all $\epsilon \in (0, \epsilon_0)$, the system $H_\epsilon$ admits a trajectory $(p, q, I, \varphi)(t)$, $t \in [0, T]$, satisfying
\[
  I(0) < I^-, \quad I(T) > I^+.
\]
\end{theorem}
\begin{remark}
		Our theorem says that for the system $\frac12p^2 + V(q) + G(I) + \epsilon h$ a diffusion trajectory whose $I$ variable varies over the range $[I^-, I^+]$ always exists after a (possibly arbitrarily small) generic potential perturbation $\epsilon g$. This is commonly referred to as \emph{generic in the sense of Ma\~n\'e} (see \cite{Man92}), which is stronger than genericity in the space $\cA_\sigma$.
\end{remark}
\section{Geometric description of diffusion for a generic perturbation}

\subsection{Scattering map and the geometric construction of diffusion}

For $\varepsilon = 0$, Hamiltonian \eqref{eq:unstable} becomes
\[
 \frac12 p^2 + V(q) + G(I),
\]
with associated equations
\begin{equation*}
\dot{p} = V'(q), \quad\dot{q} = p, \qquad
\dot{I}= 0,\quad\dot{\varphi} = \omega(I),\quad
\dot{s}=1,
\end{equation*}
so that $I$ is a constant of motion and the flow based on the homoclinic trajectory~\eqref{eq:homoclinic} has the form
\[
\Phi_0^{t}(I ,\varphi) =(p_0(t),q_0(t),I ,\varphi + t\omega(I)).
\]

For any $I\in\mathbb{R}$,
$ \widetilde{\mathcal{T}}^0_{I} = \left\{(0,0,I,\varphi,s); (\varphi,s) \in \mathbb{T}^{2}\right\}$
is an invariant 2D-torus under the flow of the system with frequency $\tilde{\omega}(I)= (\omega(I), 1)$ and is called a \emph{whiskered torus}.
For each whiskered torus $\tilde{\mathcal{T}}^0_{I}$, we have associated coincident stable and unstable 3D-manifolds called \emph{whiskers}, which we denote by
\begin{equation*}
W^0\widetilde{\mathcal{T}}^0_{I} = \left\{(p_0(\tau) , q_0(\tau) , I, \varphi, s):\tau\in\mathbb{R}, (\varphi , s)\in\mathbb{T}^2\right\}.
\end{equation*}

The union of all whiskered tori $\widetilde{\mathcal{T}}^0_{I}$
\[
\tilde{\Lambda}_0 = \left\{(0,0,I,\varphi,s):(I,\varphi,s)\,\in\,\mathbb{R}\times\,\mathbb{T}^{2}\right\}
\]
is a 3D-\emph{Normally Hyperbolic Invariant Manifold} (NHIM) with 4D-coincident stable and unstable invariant manifolds, forming a so-called \emph{separatrix}, given by
\begin{equation*}
W^0\tilde{\Lambda}_0 = \{(p_0(\tau),q_0(\tau),I,\varphi,s):\tau\in\mathbb{R}, (I,\varphi,s)\,\in\,\mathbb{R}^2\times\,\mathbb{T}^{2}\}.
\end{equation*}

For $0<\varepsilon\ll 1$, the NHIM $\tilde{\Lambda}_0$ is preserved to a NHIM $\tilde{\Lambda}_{\varepsilon}$~\cite[Sec.~4.2]{DdS08}.
Nevertheless, its stable manifold $W^{\text{s}}(\tilde{\Lambda}_{\varepsilon})$ and unstable manifold $W^{\text{s}}(\tilde{\Lambda}_{\varepsilon})$ no longer coincide, that is, the separatrix splits. The existence of this splitting can be detected by a perturbation argument in terms of the \emph{Melnikov potential}
\begin{equation}  \label{eq:melnikov}
\begin{aligned}
		\cL(I, \varphi, s) =\cL_h (I, \varphi, s)  :=
- \int_{-\infty}^\infty
& \bigl[ h(p_0(t), q_0(t), \varphi + t\,\omega(I), I, s + t)  \\
&  \quad - h(0, 0, \varphi + t\,\omega(I), I, s + t) \bigr ] dt.
\end{aligned}
\end{equation}

\begin{proposition}[DLS06]
\label{prop:homoclinic}
Given $(I,\varphi,s)\,\in\,\left[-I^{*},I^{*}\right]\,\times\,\mathbb{T}^{2}$, assume that the real function
\begin{equation}\label{eq: SM_critical_point}
\tau\,\in\,\mathbb{R}\,\longmapsto\,\mathcal{L}(I ,\varphi -\tau\,\omega(I) ,\,s-\tau)\,\in\,\mathbb{R}
\end{equation}
has a non-degenerate critical point $\tau^{*}\, =\, \tau^*(I,\varphi,s)$.
Then, for $0\,<\,\varepsilon$ small enough, there exists a unique transverse homoclinic point $\tilde{z}_{\varepsilon}$ to $\tilde{\Lambda}_{\varepsilon}$ of Hamiltonian~\eqref{eq:unstable}, which is $\varepsilon$-close to the point
$\tilde{z}^{*}(I,\varphi,s)\,=\,(p_{0}(\tau^{*}),q_{0}(\tau^{*}),I,\varphi,s)\,\in\,W^{0}(\tilde{\Lambda})$:
\begin{align*}
&\tilde{z}_{\varepsilon}=\tilde{z}_{\varepsilon}(I,\varphi,s)=(p_{0}(\tau^{*})+O(\varepsilon), q_{0}(\tau^{*})+O(\varepsilon),I,\varphi,s)\\
&\text{and }\tilde{z}_{\varepsilon}\in\,W^{u}(\tilde{\Lambda_ {\varepsilon}})\,\pitchfork\,W^{s}(\tilde{\Lambda_{\varepsilon}}).
\end{align*}
\end{proposition}

Using this Proposition we can introduce the notion of the \emph{scattering map}, which plays a central r\^ole in our mechanism for detecting diffusion.
Let $W$ be an open set of $\left[-I* , I^*\right]\times  \mathbb{T}^2$ such that the invariant manifolds of the NHIMH $\tilde{\Lambda}_{\varepsilon}$
intersect transversely along a homoclinic manifold $\Gamma_{\varepsilon}=\left\{\tilde{z}_{\varepsilon}(I,\varphi,s) , (I,\varphi,s)\in W\right\}$ and for any $\tilde{z}_{\varepsilon}\in\Gamma_{\varepsilon}$ there exists a unique $\tilde{x}_{\pm} = \tilde{x}_{\pm,\varepsilon}(I,\varphi,s)\in\tilde{\Lambda}$ such that $\tilde{z}_{\varepsilon}~\in~W_{\varepsilon}^{s}(\tilde{x}_-)\cap W_{\varepsilon}^{u}(\tilde{x}_+)$.
Let
$$H_{\pm} = \bigcup\left\{\tilde{x}_{\pm}=\tilde{x}_{\pm,\varepsilon}(I ,\varphi , s) : (I,\varphi,s)\in W\right\}.$$
The scattering map associated to $\Gamma_\epsilon$ is the map
\begin{eqnarray*}
S_{\varepsilon}: H_- & \longrightarrow & H_+\\
\tilde{x}_- &\longmapsto & S_{\varepsilon}(\tilde{x}_-) = \tilde{x}_+.
\end{eqnarray*}

Notice that the domain of definition of the scattering map depends on the homoclinic manifold chosen.
Therefore, for the characterization of the scattering maps, it is required to select the homoclinic manifold $\Gamma_\epsilon$,
which is determined by the function $\tau^*(I ,\varphi ,s)$.
Once a function $\tau^*(I,\varphi,s)$ is chosen, by the geometric properties of the scattering map, particularly its exactness, see \cite{DdS08}, the scattering map $S_{\varepsilon} = S_{\varepsilon,\tau^*}$ has the explicit form~\cite[eq. (9.9)]{DLS06}
\begin{equation}
\label{eq:SM1}
S_{\varepsilon}(I,\varphi,s) = \left(I \ + \varepsilon\partial_{\varphi}L^* + \mathcal{O}(\varepsilon^2), \varphi - \varepsilon\partial_{I} L^* +\mathcal{O}(\varepsilon^2), s \right),
\end{equation}
where $L^* = L^*(I,\varphi,s) $ is the \emph{Poincar\'e function defined by}
\begin{equation}\label{eq:PoincareFunction}
L^*(I,\varphi,s) = \mathcal{L}\left(I,\varphi -\tau^*(I,\varphi,s)\,\omega(I) , s-\tau^*(I,\varphi,s)\right).
\end{equation}

Notice that if $\tau^*(I,\varphi, s)$ is a critical point of~\eqref{eq: SM_critical_point},
$\tau^*(I,\varphi, s)-\sigma$ is a critical point of
\begin{align}\label{eq: SM_critical_point_sigma}
\tau\,\in\,\mathbb{R}\,\longmapsto&\,\mathcal{L}(I ,\varphi -(\tau+\sigma)\,\omega(I) ,\,s-(\tau+\sigma))\nonumber\\
=&
\mathcal{L}(I ,\varphi -\sigma\,\omega(I)-\tau\,\omega(I),s-\sigma-\tau)
\end{align}
Since $\tau^*(I,\varphi-\sigma\,\omega(I), s-\sigma)$ is a critical point of the right-hand side
of~\eqref{eq: SM_critical_point_sigma}, by the uniqueness
in $W$ we can conclude that
\[
\tau^*(I,\varphi-\sigma\,\omega(I), s-\sigma)=\tau^*(I,\varphi, s)-\sigma.
\]
Thus, by~\eqref{eq:PoincareFunction}, the Poincar\'e function $L^*$ satisfies
\[
L^*(I ,\varphi -\sigma\,\omega(I),s-\sigma)=L^*(I ,\varphi,s), \quad \sigma\in\R,
\]
and, in particular, for $\sigma = s$,
\[
L^*(I ,\varphi -s\,\omega(I),0)=L^*(I ,\varphi,s).
\]
In other words, the Poincar\'e function $L^*$ is invariant under the flow $\dot\varphi=\omega(I), \dot s=1$ or, equivalently, is constant along any line in the $(\varphi,s)$-plane $\varphi - s\,\omega(I)=\text{ constant}$.

These previous equations are telling us that we can reduce the variables in both the Poincar\'e function $L^*$ and the time $\tau^*$. Indeed, writing them for $\sigma=s$ and introducing the variable
\begin{equation}
\theta = \varphi - s\,\omega(I),\label{eq:def_theta}
\end{equation}
we can define the \emph{reduced Poincar\'{e} function} $\mathcal{L}^{*}$ as well as the reduced time $\bar{\tau^*}$, defined only in the variables $(I,\theta)$, as
\[
\mathcal{L}^{*}(I,\theta) := L^*(I,\varphi - s \,\omega(I)  , 0) = L^*(I,\varphi,s),
\]
\begin{align}\label{eq:tau_theta}
\bar{\tau^*}(I,\theta) := \tau^*(I,\varphi-s\,\omega(I),0)=\tau^*(I,\varphi,s) - s.
\end{align}
Moreover we also have
\begin{equation}
\label{eq:useful_reduced_poincare_function}
\mathcal{L}^*(I,\theta) = \mathcal{L}(I,\theta - \bar{\tau}^*(I,\theta)\,\omega(I), -\bar{\tau}^*(I,\theta)).
\end{equation}

Note that the variable $s$ is fixed under the scattering map~\eqref{eq:SM1}.
As a consequence, we can consider, for instance, the section
\[
		\Lambda_\epsilon = \tilde{\Lambda}_\epsilon \cap \{s = 0\}
\]
and the scattering map $S_\epsilon$ is well defined on $\Lambda_\epsilon$.

In the variables $(I,\theta)$ introduced in~\eqref{eq:def_theta}, the scattering map has the simple form
\[
		\mathcal{S}_{\varepsilon}(I,\theta) = \cS_{\varepsilon}(I, \varphi) = \left( I + \varepsilon\frac{\partial \mathcal{L}^*}{\partial\theta}(I,\theta) + \mathcal{O}(\varepsilon^2) ,
 \theta - \varepsilon\frac{\partial \mathcal{L}^*}{\partial I}(I,\theta) + \mathcal{O}(\varepsilon^2)   \right).
\]
So, up to $\mathcal{O}(\varepsilon^2)$ terms, $\mathcal{S}_{\varepsilon}(I,\theta)$ is the $-\varepsilon$ times flow of the \emph{autonomous} Hamiltonian $\mathcal{L}^*(I,\theta)$.
In particular, a finite number of iterates under the scattering map follow the level curves of $\mathcal{L}^*$ up to $\mathcal{O}(\varepsilon^2)$.

\begin{remark}
		$\cS_\epsilon$ is close to identity only when there is no phase shift, i.e., thanks to the fact that we are only dealing with an \emph{uncoupled} a priori unstable system. If we introduce a coupling term  as in \eqref{eq:phase-shift}, then
		\[
		  \cS_\epsilon(I, \theta) = (I, \theta + \gamma(I)) + \cO(\epsilon)
		\]
		where $\gamma(I)$ is the phase shift. In this case $\cS_\epsilon$ is no longer close to identity.
\end{remark}

Along this paper, both $\tau^*(I,\varphi,s)$ and $\bar{\tau}^*(I,\theta)$ will be used at our convenience.

We now introduce another map defined on the NHIM $\Lambda_\epsilon$. Since $\tilde{\Lambda}_\varepsilon$ is invariant under the Hamiltonian flow and $\Lambda_\epsilon$ is a global section, there exists a Poincar\'e map
\[
  T_\epsilon: \Lambda_\epsilon \to \Lambda_\epsilon.
\]
The map $T_\epsilon$ preserves the restriction of the symplectic form to $\Lambda_\epsilon$, and has the expansion
\[
T_\epsilon(I, \varphi)	
= (I, \varphi + \omega(I)) + \cO(\varepsilon).
\]
We call this map the \emph{inner map}.

In the sequel, we call a measure on a manifold a \emph{smooth measure} if it is given by a non-zero density multiplied by the volume form supported on an open set.

\begin{lemma}
		There exists $\epsilon_0 > 0$ such that for all $0 < \epsilon < \epsilon_0$ the following hold. There exists  $\gamma_-, \gamma_+ \subset \Lambda_\epsilon$ that are graphs over the variable $\theta$ and invariant under $T_\epsilon$, such that the $T_\epsilon$-invariant subset of $\Lambda_\epsilon$ bounded by $\gamma_-$ and $\gamma_+$ contains
		\[
				[I^-, I^+] \times \T \cap \Lambda_\epsilon.
		\]
		In particular, the dynamics of $T_\epsilon$ on $\Lambda_\epsilon$ preserves a smooth measure whose support contains $[I^-, I^+] \times \T \cap \Lambda_\epsilon$.
\end{lemma}

We now restrict $\Lambda_\epsilon$ to the compact invariant set between $\gamma_\pm$, which is still called $\Lambda_\epsilon$.

The following statement is a minor modification of the main result of \cite{GdM20},
and is our main technical tool for construction of diffusion orbits. This result is non-perturbative, so we formulate it for a general Hamiltonian $H$ with a compact invariant NHIM $\Lambda$.

\begin{proposition}\label{prop:shadowing}
		Assume the Hamiltonian $H$ admits a compact NHIM $\Lambda$ on which multiple scattering maps $\cS_j$ are defined on open subsets of $\Lambda$. Assume that the inner map $T$ preserves a smooth measure on $\Lambda$, and each $\cS_j$ maps positive (Lebesgue) measure sets to positive measure sets and zero measure sets to zero measure sets.

		Let $x_0, \dots, x_{N-1}$ be a pseudo-orbit
		\footnote{We use the term pseudo-orbit to stress that it is not an orbit of the Hamiltonian dynamics, even though it is an orbit of the polysystem (i.e. iteration under multiple maps).}
		under the iterate of either the inner map or one of the scattering maps, that is,
		\[
				x_{i+1} = f_i (x_i), \quad 1\le i \le N-1,
		\]
		where $f_i$ is either $T$ or one of the scattering maps $S_j$. Then for any $\delta > 0$,  there exists a trajectory $z: \R \to M$ of $H_\epsilon$ shadowing the pseudo-orbit $x_i$, i.e., there exists $t_0 < \cdots < t_N$ such that
		\[
		  \dist(x_i, z(t_i)) < \delta, \quad 0 \le i \le N.
		\]
\end{proposition}

Theorem 3.7 of \cite{GdM20} states that if the pseudo-orbit only involves the scattering maps, then the shadowing result holds. However, the proof in \cite{GdM20} can be adapted to prove Proposition~\ref{prop:shadowing} with minor changes, see Appendix~\ref{sec:shadowing}.

\begin{remark}
		Proposition~\ref{prop:shadowing} does not provide an estimate of ``diffusion time'', i.e., upper bound estimates on $t_i$. This is due to the lack of return time estimates for the inner dynamics. See Remark 3.15 of \cite{GdM20}.
\end{remark}

\subsection{Conditions on the Melnikov potential}

\begin{definition}\label{def:G}
Define $\cG(I^-, I^+)$ to be the set of functions $h$ in the space $\cA_\sigma(U)$ introduced in Def.~\ref{def:spaces} for which the following holds:
\begin{itemize}
		\item There exists intervals $(I_j^-, I_j^+)$, $j = 1, \dots, k$ such that
				\[
				  I_1^- < \dots < I_k^-, \quad
(I_j^-, I_j^+) \cap (I_{j+1}^-, I_{j+1}^+) \ne \emptyset, \quad j = 1, \dots, k-1,
				\]
				and
				\[
						[I^-, I^+] \subset \bigcup_j (I_j^-, I_j^+).
				\]
	\item			For each $j$, there exists a smooth mapping
				\[
						(\varphi_j^*, s_j^*): [I_j^-, I_j^+] \to \T \times \T ,
				\]
				such that for each $I \in [I_j^-, I_j^+]$, $(\varphi_j^*, s_j^*)(I)$ is a non-degenerate local minimum of the function
				\[
				  \cLh (I, \cdot, \cdot).
				\]
\end{itemize}
\end{definition}
We will show (Proposition~\ref{prop:generic}) that for a generic $h$, the associated Melnikov potential $\cL_h(I, \varphi, s)$ satisfies the conditions in Definition~\ref{def:G}. In fact, the local minima $(\varphi^*_j, s^*_j)$ can be chosen as local extension of the global minima of $\cL_h(I, \cdot, \cdot)$. Moreover, in Lemma~\ref{lem:critical-curve} we will see that they generate local minima $\theta^*_j(I)$ of $\cL^*_h(I,\cdot)$. In Figure~\ref{fig:generic-function}, we give an illustration of the position of global minima over $\T$ for a typical reduced Poincar\'e function $\mathcal{L}^*_h(I,\theta)$.

 \begin{figure}[ht]
 \centering
 \includegraphics[height=3in]{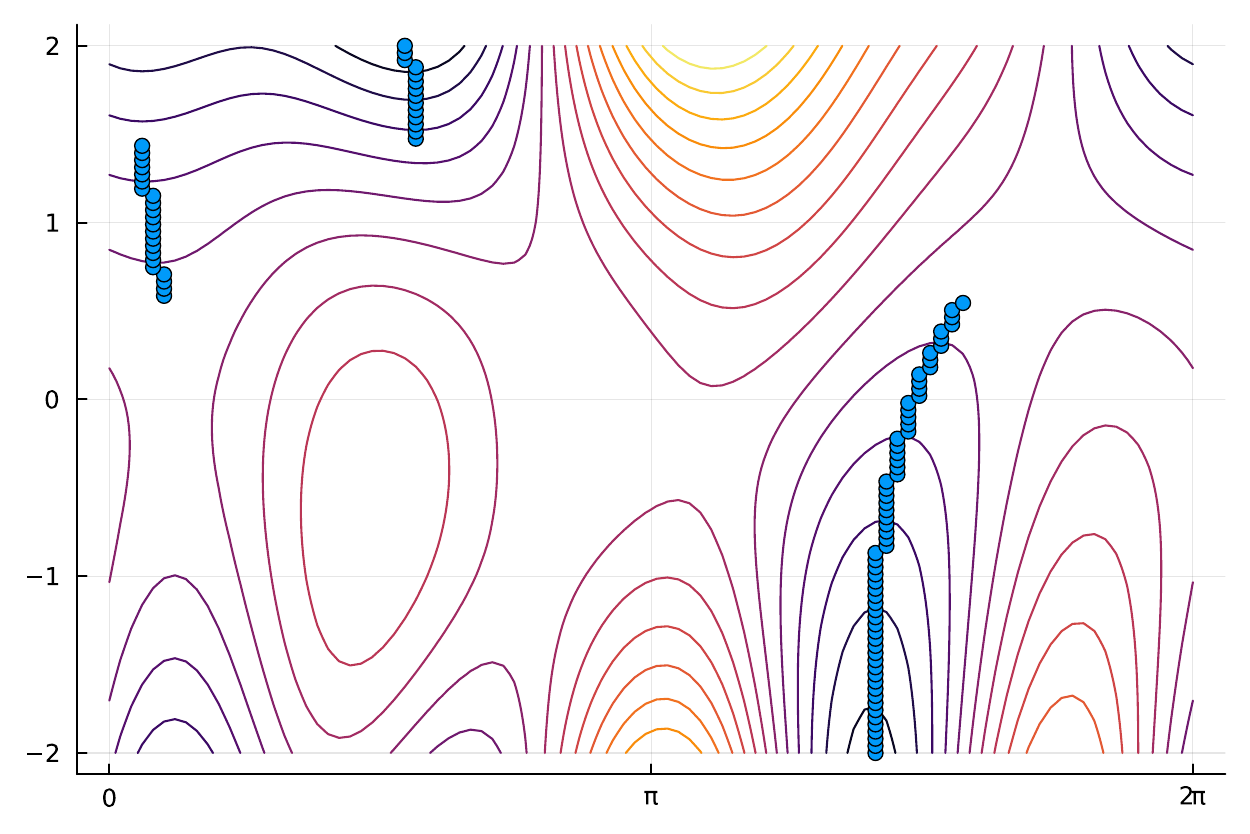}
 \caption{The levels curves and global minima of a typical reduced Poincar\'e function $\mathcal{L}^*_h(I,\theta)$ defined on $\R \times \T$. Note: the vertical coordinate is $I$ and horizontal coordinate is $\theta$.}
 \label{fig:generic-function}
\end{figure}
\begin{proposition}\label{prop:generic}
For any $h \in \cA_\sigma(U)$, there exists an open and dense set $\cG_h \subset \cP_\sigma$ such that $h + g \in \cG(I^-, I^+)$ for all $g \in \cG_h$.
\end{proposition}

\begin{proposition}
\label{prop:diffusion}
		Suppose $h \in \cG(I^-, I^+)$, then there exists $\epsilon_0(h) > 0$ such that for all $\epsilon \in (0, \epsilon_0)$, there exists an orbit $(p, q, I, \varphi)(t)$ of \eqref{eq:unstable} and $T > 0$ such that
		\[
		  I(0) < I^- < I^+ < I(T).  
		\]
\end{proposition}

Proposition~\ref{prop:generic} will be proven later. In the sequel, let us fix $h \in \cG(I^-, I^+)$ and prove existence of a diffusion trajectory.

\begin{lemma}
		Suppose $h \in \cG(I^-, I^+)$, then there exists neighborhoods $U_j(I)$ of $(\varphi^*_j(I), s^*_j(I))$ for $I \in [I_j^-, I_j^+]$, and functions $\tau_j^*(I, \varphi, s)$
		\[
				(\varphi, s) \in U_j(I) \mapsto \tau_j^*(I, \varphi, s), \quad
				\tau_j^*(I, \varphi_j^*(I), s_j^*(I)) = 0,
		\]
	depending smoothly in $(\varphi, s, I)$,  such that each $\tau_j^*$ is a non-degenerate local minimum of
\begin{equation}  \label{eq:tau-Melnikov}
		  \tau \mapsto \cLh (I, \varphi - \tau\omega(I) , s - \tau).
\end{equation}
\end{lemma}
\begin{proof}
		We need to solve the equation
\begin{equation}  \label{eq:crit-Melnikov}
\partial_\tau \left(
		\cLh (I, \varphi - \tau\omega(I), s - \tau)
\right)	 =- \left(\omega(I) \cdot \partial_\varphi \cLh 
		+ \partial_s \cLh 
\right)|_{(I, \varphi - \tau\omega(I), s-\tau)} = 0.
\end{equation}
Since $(\varphi^*_j(I), s^*_j(I))$ is a critical point of $\cLh (I, \cdot, \cdot)$, $\tau = 0$ solves \eqref{eq:crit-Melnikov} at $(\varphi^*_j(I), s^*_j(I))$.

Since $\partial^2_{(\varphi, s)}\cLh (I, \varphi_j^*(I), s_j^*(I))$ is positive definite, and
\[
  \begin{aligned}
		&	\partial_\tau
			\left(
					-\left(\omega(I) \cdot \partial_\varphi \cLh 
							+ \partial_s \cLh
					\right)|_{(I, \varphi - \tau\omega(I), s-\tau)}
			\right) \\
		& =
		-\partial_\tau \left(
				\partial_{(\varphi, s)} \cLh  \bmat{\omega(I) \\ 1} |_{(I, \varphi - \tau\omega(I), s-\tau)}
		\right) \\
		& = \bmat{\omega(I) & 1} \partial^2_{(\varphi, s)} \cLh (I, \varphi - \tau\omega(I), s - \tau) \bmat{ \omega(I) \\ 1},
  \end{aligned}
\]
we get
\begin{equation}  \label{eq:strict-min}
  \partial^2_\tau  \left(
		\cLh (I, \varphi^*_j(I) - \tau\omega(I), s^*_j(I) - \tau)
\right)	> 0.
\end{equation}
By the implicit function theorem, there exists a unique $\tau_j^*(I, \varphi, s)$, depending smoothly on $(I, \varphi, s)$ in a neighborhood of the graph $(I, \varphi^*_j(I), s^*_j(I))$ over the variable $I$, solving equation \eqref{eq:crit-Melnikov}. Moreover, $\tau^*_j$ are non-degenerate local minima for the mapping \eqref{eq:tau-Melnikov} due to \eqref{eq:strict-min}. 
\end{proof}

In view of this Lemma, on each interval $(I_j^-,I_j^+)$ and for $\theta=\varphi-s\,\omega(I)$ with $(\varphi,s)\in U_j(I)$, we can now introduce $\bar{\tau}^*_j(I,\theta)$ and $\cL^*_j(I,\theta)$ by the same formulas~(\ref{eq:tau_theta},\ref{eq:useful_reduced_poincare_function}) defining $\bar{\tau}^*(I,\theta)$ and $\cL^*(I,\theta)$, simply changing $\tau^*$ to $\tau^*_j$. We will also denote
\[
  \theta_j^*(I) = \varphi_j^*(I) -  s_j^*\omega(I).
\]

\begin{lemma}\label{lem:critical-curve}
		For each $I \in [I_j^-, I_j^+]$, $\theta_j^*(I)$ is a non-degenerate minimum of $\cL_j^*(I, \cdot)$.	
\end{lemma}
\begin{proof}
If $\tau^*=\tau^*(I,\varphi,s)$ is a critical point of~\eqref{eq:tau-Melnikov}, it satisfies
\[
\left(\omega(I) \partial_\varphi \cLh+\partial_s \cLh\right)\left(I,\varphi -\tau^* \omega(I),s-\tau^*\right)=0,
\]
so that we also have
\begin{equation}  \label{eq:crest}
\begin{aligned}
&\left(\omega(I) \partial_\varphi \cLh+\partial_s \cLh\right)\Bigr|_{\left(I,\theta -\bar{\tau}^*(I,\theta)\, \omega(I),-\bar{\tau}^*(I,\theta)\right)}\\
=&\left(\omega(I) \partial_\varphi \cLh+\partial_s \cLh\right)\Bigr|_{\left(I,\theta -\tau^*(I,\theta,0)\, \omega(I),-\tau^*(I,\theta,0)\right)}=0.\
\end{aligned}
\end{equation}

		Denote by $\id$ the identity matrix, then differentiating~\eqref{eq:useful_reduced_poincare_function} we get
		\[
\begin{aligned}
		  \partial_\theta \cL^*_j(I, \theta)
			& = \partial_\varphi \cLh  \left(\id - \partial_\theta \bar{\tau}^*_j\,\omega(I) 
            - \partial_\theta \bar{\tau}^*_j\partial_s \cLh\right) 
			\Bigr|_{(I, \theta - \bar{\tau}^*_j\,\omega(I), - \bar{\tau}^*_j) }
			\\
			& = \left(\partial_\varphi \cLh - 
           \partial_\theta \bar{\tau}^*_j
           \left(\omega(I) \partial_\varphi \cLh+\partial_s \cLh\right)\right)
            \Bigr|_{\left(I,\theta -\bar{\tau}^*_j \omega(I),-\bar{\tau}^*_j\right)}\\            
			& = \partial_\varphi \cLh  \Bigr|_{(I, \theta - \bar{\tau}^*_j\,\omega(I), - \bar{\tau}^*_j) }
\end{aligned}
		\]
		by \eqref{eq:crest}. Since $\bar{\tau}^*_j(I, \theta^*_j) = \tau_j^*(I, \varphi^*_j(I), s_j^*(I)) = 0$, we get that $\theta^*_j(I)$ is a critical point of $\cL_j^*(I, \cdot)$. Moreover, repeating the same calculation, we have
		\[
				\partial^2_{\theta \theta} \cL_j^*
				= \partial^2_{\varphi \varphi}
 \cLh  \Bigr|_{(I, \theta - \bar{\tau}^*_j\,\omega(I), - \bar{\tau}^*_j) }.
		\]
		Since $\partial^2_{\varphi \varphi} \cLh $ is positive definite at $(\varphi^*_j, s^*_j)$, $\theta_j^*$ is a non-degenerate local minimum.
\end{proof}

\begin{lemma}[Ascending ladder] \label{lem:ladder}
		There exists $\delta_0 > 0$ such that for every $\delta \in (0, \delta_0)$, there exists a sequence of curves
		\[
				\gamma_{i, j} \subset [I^-, I^+] \times \T  \quad 1 \le i \le m_j, \, 1 \le j \le k, \, l \in \{1, 2\}
		\]
		such that the following holds.
\begin{enumerate}
		\item For each $i, j$, the curve $\gamma_{i, j}$ is a smooth graph over the variable $I$, i.e.,
				\[
						\gamma_{i, j} = \{(I, f_{i, j}(I)) \st I \in [I_{i, j}^1, I_{i, j}^2]\}.
				\]
				$\gamma_{i, j}$ is a segment of the level curve of $\cL_j^*$ and the Hamiltonian flow of $\cL_j^*$ on $\gamma_{i, j}$ is increasing in the $I$ component.

		\item $I_{1, j}^1 \in [I_j^-, I_j^- + \delta]$, $I_{m_j, j}^2 \in [I_j^+ - \delta, I_j^+]$.

		\item $I_{m_j, j}^2 = I_{1, j+1}^1$ for all $j = 1, \ldots, k-1$.
\end{enumerate}
Let us also denote by $x_{i, j}^1, x_{i, j}^2$ the lower and upper end points of the curves $\gamma_{i, j}$.
\end{lemma}
\begin{proof}
		Since $\theta_j^*(I)$ is a non-degenerate local minimum of $\partial_\theta \cL^*_j(I, \cdot)$, for every $I \in [I_j^-, I_j^+]$, there exists $\phi(I) > \theta^*_j(I)$, such that
        \[
		  \partial_\theta \cL_j^*(I, \phi(I)) > 0.
		\]
		This means there is a segment $\gamma_j(I)$ of the level curve of $\cL^*_j$, on which $I$ is increasing under the Hamiltonian flow of $\cL^*_j$. By compactness, there exists a finite collection $\gamma_{i, j}$ of such level curves, $1 \le i \le m_j$ such that $\bigcup_i \pi_I (\gamma_{i, j} )\supset [I_j^-, I_j^+]$. One can adjust the collection so that $\pi_I (\gamma_{i, j}) \cap  \pi_I (\gamma_{i, j+1}) \ne \emptyset$ and $I_j^- \in \pi_I (\gamma_{1, j})$ and $I_j^+ \in \pi_I(\gamma_{m_j, j})$. It suffices to choose
		\[
\begin{aligned}
				& I_{i, j}^1 \in   \pi_I (\gamma_{i, j}) \cap  \pi_I (\gamma_{i, j-1}),
				\quad j \ge 2, \\
				& I_{i, j}^2 \in  \pi_I (\gamma_{i, j}) \cap  \pi_I (\gamma_{i, j+1}),
				\quad j \le m_j + 1,
\end{aligned}
		\]
		and $I_{1, j}^1$, $I_{2, j}^2$ satisfying (2) and (3); then truncate the curves $\gamma_{i, j}$ to the $I$ interval $[I_{i, j}^1, I_{i, j}^2]$.
\end{proof}

\begin{figure}[ht]
 \centering
 \includegraphics[height=2.5in]{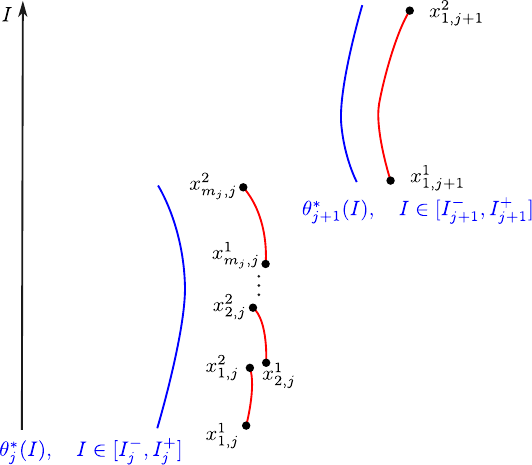}
 \caption{The ascending ladder}
\end{figure}

\subsection{Constructing the transition chain}

\begin{proposition}[Transition chain] \label{prop:transition}
		Let $h \in \cG(I^-, I^+)$, and let $\gamma_{i, j}$ be the ascending ladder constructed in Lemma~\ref{lem:ladder}. Then there is $\delta_1 > 0$ such that for every $\delta \in (0, \delta_1)$,  there exists $\epsilon_0 > 0$ depending on $h$ and $\delta$, such that for every $\epsilon \in (0, \epsilon_0)$, there exist
		\[
				y_{i, j}^l \in [I^-, I^+] \times \T, \quad   \quad 1 \le i \le m_j, \, 1 \le j \le k, \, l \in \{1, 2\},
		\]
		and $M_{i, j}, N_{i, j} \in \N$ such that the following hold.

		Let $S_{j,\epsilon}$ denote the scattering map associated to $\cL_j^*$, and $T_\epsilon$ denote the inner dynamics.
\begin{enumerate}
		\item $S_{j, \epsilon}^{M_{i, j}}(y_{i, j}^1) = y_{i, j}^2$.
		\item $T_\epsilon^{N_{i, j}}(y_{i, j}^2) = y_{i+1, j}^1$, $i = 1, \dots, m_j - 1$.
		\item $T_\epsilon^{N_{m_j, j}}(x_{m_j, j}^2) = x_{1, j+1}^1$, $j = 1, \dots, k-1$.
		\item $|x_{i, j}^l - y_{i, j}^l| < \delta$ for all $i = 1, \dots,  m_j$, $j = 1, \dots, k$, $l = 1, 2$.
\end{enumerate}
\end{proposition}

We first prove Proposition \ref{prop:diffusion} assuming Proposition \ref{prop:transition}.
\begin{proof}[Proof of Proposition \ref{prop:diffusion}]
Let $I(\cdot)$ denote the projection $(I, \theta) \mapsto I$. By Lemma \ref{lem:ladder} and Proposition \ref{prop:transition}, there exists a pseudo-orbit $\{y_{i, j}^l\}$ of the maps $S_{j, \epsilon}$ and $T_\epsilon$, such that  $I(y_{1, 1}^1) < I^-$ and $I(y_{m_k, k}^2) > I^+$. By Proposition \ref{prop:shadowing}, there exists an orbit $(q, p, I, \varphi)(t)$ of the original Hamiltonian system  and $T > 0$ such that
\[
I(0) < I^-, \quad I(T) > I^+.
\]
\end{proof}
The remaining section is dedicated to proving Proposition \ref{prop:transition}.

We have the following simple lemma about integrable twist maps.

\begin{lemma}
\label{lem:twist}
		Consider the map $T: [I^-, I^+] \times \T \to \R \times \T$ defined by
\[
  T(I, \theta) = (I, \theta + \omega(I))
\]
where $\omega'(I) > 0$ for all $I \in [I^-, I^+]$. Let $T_\epsilon: [I^-, I^+] \times \T \to \R \times \T$ be a family of maps such that
\[
		\|T_\epsilon - T\|_{C_1} \le C \epsilon
\]
for an independent constant $C$.

Then for any two $C^1$ graphs $\gamma_1$, $\gamma_2$ over the variable $I$
\[
	\gamma_1	= \{(I, f(I)) \st I \in (I_1, I_2)\}, \quad
	\gamma_2 = 	\{(I, g(I)) \st I \in (I_1, I_2)\},
\]
$x_1 \in \gamma_1$ and $\delta > 0$, there exists $\epsilon_0 > 0$, $N \in \N$, $x_2 \in \gamma_2$, $\delta' > 0$ depending on $x_1$, $\delta$ depending only on $T$ and $C$, and  $I_1, I_2$, such that for all $\epsilon \in (0, \epsilon_0)$,
\[
		T(B_\delta(x_1)) \supset B_{\delta'}(x_2).
\]
\end{lemma}
\begin{proof}
	Note that	
		\[
		  T^N(\gamma_1) =
	 \left\{
			(I, f(I) + N \omega(I)) \st I \in (I_1, I_2)
	\right\}
		\]
		has the tangent vector $(1, N \omega'(I) + f'(I))$. Let $\delta = \min \omega'(I) > 0$, we choose $N$ large enough so that
		\[
		  N \omega'(I) - \sup |f'(I)| > \frac12 N\delta > 0.
		\]
		In this case $T^N(\gamma_1)$ can be rewritten as a graph over the variable $\theta$
		\[
				T^N(\gamma_1) =  \{(h(\theta), \theta) \st \theta \in (\theta_1, \theta_2)\},
		\]
		where $|h'(\theta)| < \sup 1/(N\omega'(I) + f'(I)) < \frac{2}{N\delta}$, and
		$|\theta_2 - \theta_1| > \frac12 N\delta$. If $\frac12 N\delta > 1$, then $T^N(\gamma_1)$ must intersect $\gamma_2$ at $x_2$.

		It follows that there exists $\delta_1 > 0$ depending only on $T$ and $N$ such that
		\[
				T^N(B_{\delta}(x_1)) \supset B_{\delta_1}(x_2).
		\]
		Moreover, we can choose $\epsilon_0$ small enough such that for all $\|T_\epsilon - T\|_{C^1} < C\epsilon_0$, we have
		\[
				T_\epsilon^N(B_\delta(x_1))   \supset B_{\delta_1/2}(x_2).
		\]
\end{proof}

\begin{proof}[Proof of Proposition~\ref{prop:transition}]
		Let $\gamma_{i,j}$, $1 \le i \le m_j$, $1 \le j \le k$, $l = 1, 2$ be the ascending ladder constructed in Lemma~\ref{lem:ladder}.

		\textbf{Claim}. There exists $\delta_0 > 0$ such that for any $\delta \in (0, \delta_0)$, there exists $\epsilon_0 > 0$ and $C > 1$ depending only on the perturbation $h$ and $\delta$, such that for all $1 \le i \le m_j$, $1 \le j \le k$, any $x_1, x_2 \in \gamma_{i, j}$ with $\pi_I(x_1) < \pi_I(x_2)$, and $\epsilon < \epsilon_0$ there exists $M \in \N$ depending on $x_1, x_2, \delta, \epsilon$ such that
		\[
				S_{\epsilon, j}^{M}(B_\delta(x_1)) \supset B_{\delta/C}(x_2),
		\]
		where $B_r(x)$ denote the ball of radius $r$ at $x$.

		\textbf{Proof of Claim}.	Let $\phi^t_j$ denote the Hamiltonian flow defined by the Hamiltonian $\cL_j^*$. By construction, there exists $t > 0$ such that
		\[
				\phi^t_j(x_1)  = x_2
		\]
		There exists $c > 0$ (depending only on $\cL_j$, which depends only on $h$) and a local flow-box coordinate
		\[
				\chi_{i, j}:  (-c, t + c) \times (-c, c) \to \T \times \R
		\]
		such that $x_1 = \chi_{i, j}(0, 0)$ and
		\[
				\phi^t \circ \chi_{i, j}(s, m) = \chi_{i, j}(s + t, m).
		\]
		In particular, $x_2 = \chi_{i, j}(0, t)$. Let $M = \floor{t/\epsilon}$. Since $S_{\epsilon, j} = \phi_{i, j}^\epsilon + O(\epsilon^2)$, $S_{\epsilon, j}^M = \phi_{i, j}^{\epsilon t} + O(\epsilon)$,
\[
				\chi_{i, j}^{-1} \circ S_{\epsilon, j}^{M} \circ \chi_{i, j}(s, m)
				= (s + t, m) + O(\epsilon).
\]
For any $r \in (0, c)$ and $\epsilon_0$ small enough, we have
\[
    \chi_{i, j}^{-1} \circ S_{\epsilon, j}^{M} \circ \chi_{i, j}(B_r(s, m))
		\supset B_{r/2}(s + t, m).
\]
Let $C_1 = \max\left\{ \sup \|D\chi_{i, j}\|, \sup\|D\chi^{-1}\|\right\}$ which depends only on the perturbation $h$. Then if $C_1 \delta \in (0, c)$,
\[
\begin{aligned}
	  S_{\epsilon, j}^{M} (B_\delta(x_1) )
	& \supset
	\chi_{i, j} \circ \chi_{i, j}^{-1} \circ S_{\epsilon, j}^{M} \circ \chi_{i, j}(B_{\delta/C_1}(s, m)) \\
	& \supset \chi_{i, j}(B_{\delta/(2C_1)}(s + t, m))
	\supset B_{\delta/(2C_1^2)}(x_2).
\end{aligned}
\]
This conclude the proof of \textbf{Claim}.

We now construct a sequence of balls
\[
		B_{\delta_{i, j}^l}(z_{i, j}^l) \subset B_\delta(x_{i, j}^l),
\]
and $M_{i, j}$, $N_{i, j} \in \N$ such that
\begin{equation}  \label{eq:S-image}
		S_{\epsilon, j}^{M_{i, j}}(B_{\delta_{i, j}^1}(z_{i, j}^1))
		\supset B_{\delta_{i, j}^2}(z_{i, j}^2), \quad
		1 \le i \le m_j, \quad 1 \le j \le k,
\end{equation}
and
\[
		T_\epsilon^{N_{i, j}}(B_{\delta_{i, j}^2}(z_{i, j}^2))  \supset B_{\delta_{i^*, j^*}^1}(z_{i^*, j^*}^1),
\]
where $(i^*, j^*) = (i+1, j)$ if $i < m_j$ and $(i^*, j^*) = (1, j+1)$ if $i = m_j$ and $j < k$. This implies existence of $y_{i, j}^l \in B_{\delta_{i, j}^l}(z_{i, j}^l)$ satisfying the conclusion of our Proposition.

We choose $\delta_{1, 1}^1 = \delta$ and $z_{1, 1}^1 = x_{1, 1}^1$. Constructing by induction, we assume $\delta_{i, j}^1$ and $z_{i, j}^1$ are already chosen. Choose $z_{i, j}^2 = x_{i, j}^2$ and apply \textbf{Claim} to get existence of $\epsilon_0 > 0$ and $M_{i, j}^\epsilon \in \N$ and $\delta_{i, j}^2$ such that \eqref{eq:S-image} holds for all $\epsilon \in (0, \epsilon_0)$.

The curves $\gamma_{i, j}$ and $\gamma_{i^*, j^*}$ are given by graphs over the variable $I$
\[
		\gamma_{i, j} = \{(I, f(I)) \st I \in [I_{i, j}^1, I_{i, j}^2]\}  , \quad
		\gamma_{i^*, j^*} = \{(I, g(I)) \st I \in [I_{i^*, j^*}^1, I_{i^*, j^*}^2]\},
\]
where $I_{i, j}^2 = I_{i^*, j^*}^1$. Extending these curves along the level curves of $\cL^*_j$ and $\cL_{j^*}^*$, we may assume that the extended curves overlap in $I$ direction for some interval
\[
		I \in [I_{i, j}^2, I_{i, j}^2 + c].
\]
Applying Lemma~\ref{lem:twist}, by possibly reducing $\epsilon_0$, there exists $N \in \N$, $z_{i^*, j^*}^1 \in \gamma_{i^*, j^*}$, and $\delta_{i^*, j^*}^1 > 0$ such that for all $\epsilon \in (0, \epsilon_0)$,
\[
		T^N(B_{\delta_{i, j}^2}(z_{i, j}^2)) \supset B_{\delta_{i^*, j^*}^1}(z_{i^*, j^*}^1).
\]
We can continue the construction until all the balls are defined. Note that we need to reduce $\epsilon_0$ up to finitely many times, but it will be well defined at the end of the construction.
\end{proof}

\section{Genericity}

We prove Proposition~\ref{prop:generic} in this section. The proof follows a standard parametric transversality argument. In Section~\ref{sec:transversality}, we prove that the map $g \mapsto \cL_{h + g}$ is satisfies a transversality condition. This transversality allows us to generically avoid certain degeneracies of derivatives.

\subsection{Transversality of parameters} \label{sec:transversality}

Denote $x = (I, \varphi, s)$, and $X = \R \times \T \times \T$.
The spaces of functions $\cP_\sigma$ and $\cA_\sigma$ are introduced in Def.~\ref{def:spaces}. 

\begin{proposition}\label{prop:melnikov-transvserse}
At every  $x \in \R \times \T \times \T$, the mapping
\begin{equation}  \label{eq:L-transverse}
		E(g) := g \mapsto
		\bmat{\partial_{(s, \varphi)} \\ \partial^2_{(s, \varphi)} \\ \partial^3_{(s,\varphi)}} \cL_g(x),
		\quad
		 \cP_\sigma \to \R^{2 + 3 + 4}
\end{equation}
is a surjection.
\end{proposition}

Before the proof, we state some Lemmas.
\begin{lemma}\label{lem:smooth}
		The mapping
\[
		(x, h) \mapsto \cLh (x), \quad
		X \times \cA_\sigma \to \R
\]
is $C^\infty$ in both $x$ and $h$. Moreover, for every $l > r \ge 0$, the mapping $h \mapsto \cLh $ is well defined as a mapping $C^l \mapsto C^l$, and there exists a constant $M_r > 0$ such that
\[
		\|\cL_{h_1} - \cL_{h_2}\|_{C^r} \le M_r \|h_1 - h_2\|_{C^{r+1}}.
\]
\end{lemma}
\begin{proof}
		Recall \eqref{eq:melnikov}:
\begin{equation}  \label{eq:Lh-def}
				\begin{aligned}
						\cLh (I, \varphi, s)  = - \int_{-\infty}^\infty
						& \bigl[ h(p_0(t), q_0(t), \varphi + \omega(I)t, I, s + t)  \\
						&  \quad - h(0, 0, \varphi + \omega(I)t, I, s + t) \bigr ] dt.
				\end{aligned}
\end{equation}
		By assumption $(p_0(t), q_0(t))$ converges exponentially fast to $(0, 0)$ as $t \to \pm \infty$, the integral \eqref{eq:Lh-def} converges absolutely and uniformly over all $x \in X$, with the estimate
		\begin{equation}  \label{eq:Lh-bound}
				\sup_{x \in X} |\cLh (x)| 
				\le C_r \sup_{x \in X} |\partial_{(q, p)} h| 
		\end{equation}
		for some $C_r > 0$.

		Therefore for any multi-index $\alpha \in \N^3$, 
		\[
				\partial^\alpha_x  \cLh 
				= \cL_{\partial^\alpha_x h}.
		\]
		Since $h$ is analytic, $\cLh $ is $C^\infty$ in $x$. $\cLh $ is $C^\infty$ in $h$ since it is linear in $h$. 

		Moreover, using \eqref{eq:Lh-bound}, we have
		\[
				\|\cLh \|_{C^r} = \sup_{|\alpha| \le r} \sup_x|\partial^\alpha_x \cLh |
				\le \sup_{|\alpha| \le r} \sup_x |\cL_{\partial^\alpha_x h}|
				\le M_r \|h\|_{C^{r+1}}
		\]
		for some $M_r > 0$ depending on $C_r$.
\end{proof}

Given $r \ge 1$, let
\[
		\Delta_r = \{j, k \in \Z \st j, k \ge 0, \, j + k \le r\}.
\]

\begin{lemma}\label{lem:control-derivatives}
		Given $r \ge 1$, $(I_0, s_0, \varphi_0) \in X$, and
		\[
				\{c_{j, k} \in \R \st (j, k) \in \Delta_r\},
		\]
		for every $\delta > 0$, there exists $g \in \cP_\sigma$ such that
		\[
				|\partial_\phi^j \partial_s^k \cL_g(I_0, s_0, \varphi_0) - c_{j, k}| < \delta, \quad \forall j, k \in \Delta_r.
		\]
\end{lemma}
\begin{proof}
Let $G_0(\varphi, s)$ be a smooth function that satisfies
\[
		\partial_\varphi^j \varphi_s^k G(\varphi_0, s_0) = c_{j, k}, \quad \forall j, k \in \Delta_r.
\]

In a neighborhood $U$ of $q_0(0)$ we consider a coordinate change $\phi: U \to \R^d$ such that $\varphi(q_0(t)) = te_1$ for all $t$ such that $q_0(t) \in U$, where $e_1$ is the first coordinate vector. Let $a > 0$ be small enough such that $[-a, a]^d \subset \phi(U)$, and given $\delta_1 > 0$ let $\eta_1$ be a smooth function supported on $(-a/2, a/2)$ such that
\[
  \left|
	\int \eta_1(t) f(t) dt - f(0)
	\right| < \delta_1 \|f\|_{C^r}
\]
for any test function $f$. Let $\eta$ be a smooth function such that $\eta(x_1, \dots, x_d) = \eta_1(x_1)$ on $[-a/2, a/2]^d$ and supported on $[-a, a]^d$, then define
\[
		R_0(q) = \eta \circ \varphi(q).
\]
Let $f(t)$ be a smooth test function, we have
\[
		\left |\int_{-\infty}^\infty R_0(q_0(t)) f(t) dt - f(0) \right|
		= \left |\int_{-\infty}^\infty \eta_1(t e_1) f(t) dt - f(0) \right|
		\le \delta_1 \|f\|_{C^r}.
\]

Define $g_0(q, \varphi, t) = R_0(q) G_0(\varphi, s)$, we get
\[
\begin{aligned}
		& \quad |\partial_\varphi^j \partial_s^k \cL_g(I_0, \varphi_0, s_0)	- c_{j, k}| \\
		& =
		\left|
		\int_{-\infty}^\infty (  R(q_0(t)) - R(0) ) \partial_\varphi^j \partial_s^k G(\varphi_0 + \omega(I) t, s_0 + t) dt - c_{j, k}
\right| \\
		& =
		\left|
		\int_{-\infty}^\infty R(q_0(t))\partial_\varphi^j \partial_s^k G(\varphi_0 + \omega(I)t, s_0 + t) dt - c_{j, k}
\right|  \le \delta_1 C,
\end{aligned}
\]
where $C$ is a constant that depends only on $G$. We now choose $\delta_1$ so that $\delta_1 C < \delta/2$.

Since $\cP_\sigma$ is dense in $C^{r+1}(\T^3)$, by Lemma~\ref{lem:smooth}, we can choose $g \in \cP_\sigma$ such that
\[
		\|\cL_g - \cL_{g_0}\|_{C^r} < \delta/2.
\]
This choice of $g$ verifies the claim of our lemma.
\end{proof}

\begin{proof}[Proof of Proposition~\ref{prop:melnikov-transvserse}]
		We can represent the codomain of \eqref{eq:L-transverse} by
		\[
				\{c_{j, k} \st j,k \in \Delta_r \setminus \{(0, 0)\}\}.
		\]
		Let $\bfc_1, \dots, \bfc_{10}$ be a basis be the above space, then by Lemma~\ref{lem:control-derivatives}, for any $\delta > 0$, there exists $g_1, \dots, g_{10} \in \cP_\sigma$ such that
		\[
		  \|E(g_i) - \bfc_i\|_\infty < \delta
		\]
		for all $1 \le i \le 10$. It suffices to choose $\delta$ small enough so that $E(g_i)$ still form a basis.
\end{proof}

\subsection{Generic property of critical points}

We proceed to prove Proposition~\ref{prop:diffusion}.

\begin{lemma}\label{lem:non-deg-min}
		There exists an open and dense subset $\cG(I^-, I^+) \subset \cP_\sigma$, such that for all $I \in [I^-, I^+]$ and $g \in \cG(I^-, I^+)$, the function
		\[
				(s, \varphi) \mapsto \cL_{h + g}(I, s, \varphi), \quad
				\T^2 \to \R
		\]
		does not admit any degenerate local minima or maxima.
\end{lemma}
\begin{proof}
		Let $\bS^1 = \{v \in \R^2 \st \|v\| = 1\}$, consider the mapping
\[
  F(x, v, g) =
	 \begin{bmatrix}
			\partial_\varphi \cL_{h + g}(x) \\
	\partial_s \cL_{h + g}(x) \\
	\partial_{(s, \varphi)}^2 \cL_{h + g}(x)  v \\
	\partial^3_{(s, \varphi)} \cL_{h + g}(x) [v, v, v]
	\end{bmatrix}, \quad
	X \times \bS^1 \times \cP_\sigma \to \R^5.
\]
Since $F$ is linear in $g$,
\[
  \partial_g F(x_0, v_0, g_0)(\delta g) = F(x_0, v_0, \delta g).
\]

By Proposition~\ref{prop:melnikov-transvserse}, at every $(x_0, v_0, g_0) \in \R \times \bS^1 \times \cP_\sigma$, there exists a finite dimensional subspace $E \subset \cP_\sigma$, such that the mapping
\[
		\delta g \mapsto \partial_g F(x_0, v_0, g_0)(\delta g) = F(x_0, v_0, \delta g), \quad
		E \to \R^5
\]
is a surjection. Moreover, by the smoothness of the mapping $F$ in $(x_0, v_0, g_0)$, the same is true for all $(x, v, g)$ contained in a neighborhood $U_{\mathrm{loc}}$ of $(x_0, v_0, g_0) \in X \times \bS^1 \times E$. By the implicit function theorem, the set
\[
		W_{\mathrm{loc}} = \{(x, v, g) \st F(x, v, g) = 0\}
\]
is a submanifold of codimension $5$ in $U_{\mathrm{loc}}$. Let $\pi_g$ be the projection into the $g$ component, then by Sard's Theorem, the set of regular values $g \in \pi_g(U_{\mathrm{loc}})$, denoted $\cG_{\mathrm{loc}}$, is a residual set in $U_{\mathrm{loc}}$. Since $d\pi_g$ has at most co-rank $4$, $g_1$ is a regular value implies $\pi_g^{-1}(g_1) \cap W_{\mathrm{loc}} = \emptyset$, which means $F(x, v, g_1) \ne 0$ for all $g_1 \in \cG_{\mathrm{loc}}$ and $(x, v, g_1) \in U_{\mathrm{loc}}$.

Let $U_i$ be a countable covering of the separable space $X \times \bS^1 \times \cP_\sigma$ and $\cG_i \subset \cP_\sigma$ be the corresponding local residual sets. Then \[
  \cG = \bigcup_i \cG_i
\]
is a residual subset of $\cP_\sigma$, such that for every $g \in \cG$ and $(x, v) \in X \times \bS^1$,
\[
  F(x, v, g) \ne 0.
\]
Finally, set
\[
		\cG(I^-, I^+) = \{g \st F(x, v, g) \ne 0, \text{ for all } x, v \in [I^-, I^+] \times \T \times \T \times \bS^1\},
\]
this set is open due to compactness in $(x, v)$. It is also dense since $\cG(I^-, I^+) \supset \cG$.

To prove our lemma, let $g \in \cG(I^-, I^+)$ and let $x_0 = (I_0, s_0, \varphi_0)$ be a degenerate critical point of $\cL_{h + g}$, i.e., $\partial_{(s, \varphi)} \cL_{h + g} = 0$ and there exists $v \in \bS^1$ such that
\[
		\partial^2_{(s, \varphi)} \cL_{h + g}(x_0) v = 0.
\]
Since $F(x, v, g) \ne 0$, we must have
\[
  	\partial^3_s{(s, \varphi)} \cL_{h + g}(x) [v, v, v]  \ne 0.
\]
We conclude that the function $t \mapsto \cL_{h + g}(x_0 + t(0, v))$ is an open mapping near $t = 0$, hence $(s_0, \varphi_0)$ cannot be a local minimum or maximum of $\cL_{h+g}(I_0, \cdot, \cdot)$.
\end{proof}

\begin{proof}[Proof of Proposition~\ref{prop:diffusion}]
		By Lemma~\ref{lem:non-deg-min}, if $g \in \cG(I^-, I^+)$, then for all $I \in [I^-, I^+]$ the function $\cL_{h + g}(I, \cdot, \cdot)$ has no degenerate minima or maxima.

		Fix $g \in \cG(I^-, I^+)$, and let $M_g \subset \T \times \T$ denote the set of all global minima of $\cL_{h + g}(I, \cdot, \cdot)$ for some $I \in [I^-, I^+]$. Since all the minima are non-degenerate, they are isolated on each level $\{I\} \times \T \times \T$. Moreover, any non-degenerate local minima can be extended smoothly to a neighborhood of $I$. We conclude that $M_g$ is contained in a finite family of curves of non-degenerate local minima $\{(I, \varphi_J^*(I), s_J^*(I)) \st I \in J\}$ where $J \in \cJ$ is a finite collection of open intervals, and $\bigcup_{J \in \cJ} J \supset [I^-, I^+]$. We now select the intervals $[I_j^-, I_j^+]$ inductively as such:
\begin{itemize}
		\item Let $(I_1^-, I_1^+)$ be any interval in $\cJ$ what contains $I^-$.
		\item If intervals $(I_j^-, I_j^+)$ are selected for $j \le m$, pick any interval $(a, b)$ in $\cJ$ that contains $I_m^+$. Set $I_{m+1}^- = \max \{a, (I_m^- + I_m^+)/2 \}$ and $I_{m+1}^+ = b$.
\end{itemize}
\end{proof}

\appendix

\section{Shadowing pseudo-orbits using Poincar\'e recurrence}
\label{sec:shadowing}

In this section we describe the proof of Proposition~\ref{prop:shadowing}. This Proposition has minor differences from Theorem 3.7 of \cite{GdM20} by allowing the iterates of inner dynamics. Our assumption on the measure preserving properties of $T$ and $S_j$ is also stronger than that of \cite{GdM20} but nevertheless satisfied for a priori unstable systems. The proof follows the main ideas of \cite{GdM20}. 

\begin{lemma}[Lemma 3.2 of \cite{GdM20}]
\label{lem:shadowing-primitive}
		Suppose $\Lambda_\epsilon$ is an NHIM of the Hamiltonian $H$ and $T$ is the associated inner map. Assume that there exists a finite family $S_j$, $1 \le j \le m$, of scattering maps defined on open subsets of $\Lambda_\epsilon$. Assume that $\Lambda_\epsilon$ is compact.

		Then for every $\delta > 0$, there exists two family of functions $n_i^*: \N^i \to \N$ and $m_i^*: \N^{2i + 1} \times N^{i+1} \to \N$ such that for every pseudo-orbit $y_i$, $0 \le i \ne N-1$, of the form
		\[
				y_{i+1} = T^{m_i} \circ S_{j_i} \circ T^{n_i}
		\]
		such that
		\[
				n_i \ge n_i^*(j_0, \dots, j_i), \quad m_i \ge m^*(n_0, \dots, n_i, m_0, \dots, m_{i-1}, j_0, \dots, j_i),
		\]
		there exists and orbit $z(t)$ of the Hamiltonian flow and $t_0 < \dots < t_N$ such that
		\[
				d(z(t_i), y_i) < \delta.
		\]
\end{lemma}

To prove Proposition~\ref{prop:shadowing}, given any pseudo-orbit of $T$ and $S_j$, we can use Poincar\'e recurrence to attach arbitrarily long excursions of iterates of $T$. This allows us to construct a new pseudo-orbit that satisfies the requirement of Lemma~\ref{lem:shadowing-primitive}.

In the rest of the section,any  unspecified ``measure'' stands for Lebesgue measure.

\begin{proof}[Proof of Proposition~\ref{prop:shadowing}]
		First, rewrite the pseudo-orbit in the following form
		\[
				x_{i + 1} = T^{s_i} \circ S_{j_i} (x_i) =: f_i(x_i), \quad 0 \le i \ne N-1
		\]
		where $s_i \ge 0$. Note that each $f_i$ maps positive measure sets to positive measure sets, and zero measure sets to zero measure sets. Since each $f_i$ are continuous, then there exists open balls $B_i$ with $x_i \in B_i \subset B_{\delta/2}(x_i)$ such that
		\[
				f_i(B_i) \subset B_{i+1}.
		\]
		Let $n_i^*, m_i^*$ be the functions given by Lemma~\ref{lem:shadowing-primitive} depending on $\delta/2$ and $j_i$.

		By the Poincar\'e recurrence theorem, there exists $n_0 \ge n_0^*(j_0)$ such that both
		\[
			E_0' = \{x \in B_0 \st T^{k_0}(x) \in B_0 \},
		\]
		and $T^{n_0}(E_0')$ have positive measure. Since
		\[
			f_0 \circ T^{n_0} (E_0'),
		\]
		has positive measure, there exists $l_0 \ge  m_0^*(n_0, j_0) - s_0$ such that
		\[
				F_0' = \{x \in f_0 \circ T^{n_0} (E_0') \st T^{l_0}(x) \in f_0 \circ T^{n_0} (E_0')\}
		\]
		has positive measure and so does $F_0 = T^{l_0}(F_0') \subset f_0 \circ T^{n_0} (E_0') \subset B_1$.  Set $m_0 = l_0 + s_0$ and $g_0 = T^{l_0} \circ f_0 \circ T^{n_0} = T^{m_0} \circ S_{j_0} \circ T^{n_0}$, and
		\[
				E_0 = g_0^{-1}(F_0) \cap E_0'.
		\]
		Then $g_0(E_0) = T^{l_0}(F_0') = F_0$ and both $E_0$ and $F_0$ have positive measure.

		Inductively, suppose $E_i$, $F_i$, $g_{i} = T^{m_{i}} \circ S_{j_{i}} \circ T^{n_{i}}$ are constructed such that
		\[
				n_{i} \ge n_{i}^*(j_0, \dots, j_{i}), \quad
				m_{i} \ge m_{i}^*(n_0, \dots, n_{i}, m_0, \dots, m_{i-1}, j_0, \dots, j_{i}),
		\]
		satisfies
		\[
		  E_i \supset \cdots \supset E_0,
		\]
		\[
				F_i = g_{i} \circ \cdots \circ g_0 ( E_i ) \subset B_{i+1}
		\]
		and $E_i$, $F_i$ have positive measure. See Figure~\ref{fig:shadowing} for a diagram of the sets involved in the construction.

		We set
		\[
				E_{i+1}'  = \{x \in F_i \st T^{n_{i+1}}(x) \in F_i\},
		\]
		where $n_{i+1} \ge n_{i+1}^*(j_0, \dots, j_{i+1})$ is chosen such that both $E_{i+1}'$ and $T^{n_{i+1}}(E_{i+1})$ has positive measure. Let $l_{i+1} \ge m_{i+1}^*(n_0, \dots, n_{i+1}, m_0, \dots, m_i, j_0, \dots, j_i) - s_{i+1}$ such that
		\[
				F_{i+1}' = \{x \in f_{i+1} \circ T^{n_{i+1}}(E_{i+1}') \st T^{n_{i+1}}(x) \in f_{i+1} \circ T^{n_{i+1}}(E_{i+1}')\}
		\]
		$m_{i+1} = l_{i+1} + s_{i+1}$, $F_{i+1} = T^{l_{i+1}}(F_{i+1}') \subset f_{i+1} \circ T^{n_{i+1}}(E_{i+1}') \subset B_{i+2}$. Our construction ensures
		\[
				g_{i+1} \circ \cdots g_0(E_i) = g_{i+1}(F_i) \supset	g_{i+1}(E_{i+1}')  \supset F_{i+1},
		\]
		therefore
		\[
				E_{i+1} := \left( g_{i+1} \circ \cdots g_0(E_i) \right)^{-1} (F_{i+1}) \cap E_i
		\]
		is mapped onto $F_{i+1}$ by $g_{i+1} \circ \cdots g_0(E_i)$.

\begin{figure}[ht]
 \centering
 \includegraphics[height=2.5in]{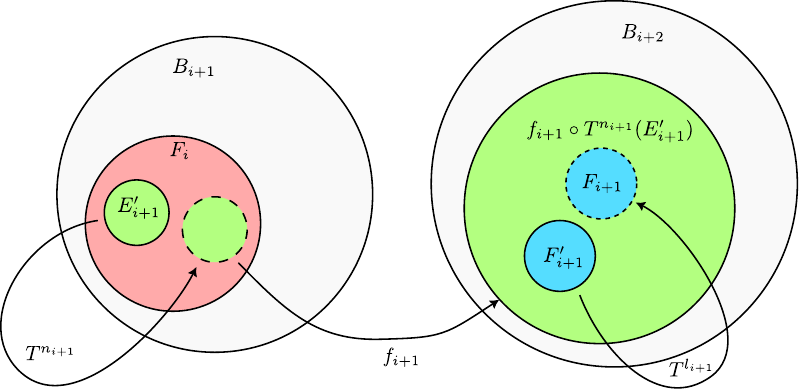}
 \caption{Diagram for the choice of the sets $E_i$, $F_i$.}\label{fig:shadowing}
\end{figure}

		Apply the construction for $0 \le i \le N-1$, we obtain
		\[
				g_i \circ \cdots g_0 (E_{N-1}) \subset g_i \circ \cdots g_0(E_i) = F_i \subset B_{i+1},
		\]
		for all $0 \le i \le N-1$. Pick any $y_0 \in E_N$, then
		\[
				y_{i+1} = g_i (y_i)
		\]
		defines a pseudo-orbit satisfying the condition of Lemma~\ref{lem:shadowing-primitive} and
		\[
		  d(y_i, x_i) < \delta/2.
		\]
		Apply Lemma~\ref{lem:shadowing-primitive}, we get a real orbit shadowing $y_i$ within $\delta/2$ distance, which shadows $x_i$ within $\delta$ distance.
\end{proof}

\bibliographystyle{abbrv}
\bibliography{analytic-a-priori-unstable}

\end{document}